\documentclass{article}

\usepackage{amsmath,amssymb,amsthm, amscd, verbatim, setspace}
\usepackage{graphicx} 
\usepackage{url, hyperref}
\usepackage[section]{placeins}
\usepackage{tabu}
\usepackage{enumerate}
\usepackage{mathtools}
\usepackage[colorinlistoftodos,prependcaption,textsize=tiny]{todonotes}

\usepackage[utf8]{inputenc}
\usepackage[T1]{fontenc}

\makeatletter
\newtheorem*{rep@theorem}{\rep@title}
\newcommand{\newreptheorem}[2]{%
\newenvironment{rep#1}[1]{%
 \def\rep@title{#2~\ref{##1}}%
 \begin{rep@theorem}}%
 {\end{rep@theorem}}}
\makeatother

\theoremstyle{plain}
\newtheorem{theorem}{Theorem}[section]

\newtheorem{thm}[theorem]{Theorem}
\newreptheorem{thm}{Theorem}

\newtheorem{lem}[theorem]{Lemma}
\newtheorem{prop}[theorem]{Proposition}
\newreptheorem{prop}{Proposition}
\newtheorem{cor}[theorem]{Corollary}
\newreptheorem{cor}{Corollary}
\newtheorem{conj}[theorem]{Conjecture}
\newreptheorem{conj}{Conjecture}

\theoremstyle{definition}

\theoremstyle{definition}
\newtheorem{defn}{Definition}

\theoremstyle{remark}

\numberwithin{equation}{section}

\makeatletter
\def\subsubsection{\@startsection{subsubsection}{3}%
  \z@{.5\linespacing\@plus.7\linespacing}{.1\linespacing}%
  {\normalfont\itshape}}
\makeatother

\newcommand{\C}{\mathbb{C}}

\newcommand{\NN}{\mathbb{N}}

\newcommand{\s}{{\mathsf{s}}}
\newcommand{\D}{{\mathsf{D}}}

\usepackage{graphicx}
\usepackage{amssymb}
\usepackage{amsthm}

\begin{document}

\title{Short Zero-Sum Sequences Over Abelian $p$-Groups of Large Exponent}

\author{Sammy Luo}

\maketitle

\begin{abstract}
Let $G$ be a finite abelian group with exponent $n$. Let $\eta(G)$ denote
the smallest integer $\ell$ such that every sequence over $G$ of length
at least $\ell$ has a zero-sum subsequence of length at most $n$. We determine the precise value of $\eta(G)$ when $G$ is a $p$-group whose Davenport constant is at most $2n-1$. This confirms one of the equalities in a conjecture by Schmid and Zhuang from 2010.

\end{abstract}


\section{Introduction}

Throughout this paper, let $G$ be an additive finite abelian group. By the Fundamental Theorem of Finite Abelian Groups, we can uniquely write
$$G\cong C_{n_1} \oplus \cdots \oplus C_{n_r},$$
where $1<n_1\mid\cdots\mid n_r$. Here $r=r(G)$ is called the \emph{rank} of $G$, and $n_r=\exp(G)$ is called the \emph{exponent} of $G$. The exponent is, for a finite abelian group, the largest order attained by an element of $G$.

We consider sequences of elements of $G$. A \emph{sequence of length $k$ over $G$} is a tuple $S=(g_1,\dots,g_k)$ of $k$ elements of $G$, where repetition is allowed and the order of elements is disregarded. Algebraically, we can consider sequences over $G$ as elements of the free abelian monoid over $G$, a viewpoint that facilitates many algebraic applications. Accordingly, we can write sequences in multiplicative notation; for example, $S=g_1\cdot \ldots \cdot g_k$. We say that $S$ is a \emph{zero-sum sequence (ZSS)} if $\sum_{i=1}^k g_i=0$. A zero-sum sequence of length at most $\exp(G)$ is called a \emph{short zero-sum sequence}.

One important class of problems on zero-sum sequences asks for extremal conditions on a sequence $S$ guaranteeing the containment of a zero-sum subsequence with specified properties, such as having a length satisfying various constraints. The exploration of this class of problems centers on a family of invariants associated with a group $G$, called the \emph{zero-sum invariants} $\s_L(G)$.
\begin{defn}
    For $L\subseteq \NN^+$, let $\s_L(G)$ be the minimum length $\ell\in \NN^+$ such that every sequence $S$ over $G$ of length at least $\ell$ contains a zero-sum subsequence of some length $t\in L$. Let $\s_L(G)=\infty$ if such a length $\ell$ does not exist.    
\end{defn}

A few examples of zero-sum invariants are of particular interest.
\begin{itemize}
    \item $\s_{\NN^+}(G)=\D(G)$ is called the \emph{Davenport constant}.  
    \item $\s_{\{\exp(G)\}}(G)=\s(G)$ is called the \emph{Erd\H{o}s-Ginzburg-Ziv (EGZ) constant}. 
    \item $\s_{[1,\exp(G)]}(G)=\eta(G)$ is called the \emph{$\eta$-invariant}.
\end{itemize} 

The study of these invariants has applications ranging from non-unique factorizations \cite{nonunique, addgpNonunique} to zero-sum Ramsey theory \cite{zsramsey} and finite geometry \cite{zssBounds}.

The precise values of $\s(G)$ and $\eta(G)$ are known only for some very special types of groups $G$. For example, for groups of rank at most $2$, we have the following result.

\begin{thm}[Geroldinger, Halter-Koch, {\cite[Theorem~5.8.3]{nonunique}}]
\label{thm:rank2}
Let $G=C_{n_1}\oplus C_{n_2}$ with $1\leq n_1|n_2$. Then
\[\s(G)=2n_1+2n_2-3\text{ and }\eta(G)=2n_1+n_2-2.\]
\end{thm}

In other cases, some bounds on $\s(G)$ and $\eta(G)$ are known. Schmid and Zhuang \cite{conjSrc} obtained the following result for the case of a $p$-group with $\D(G)\leq 2\exp(G)-1$. We will call such a group a \emph{$p$-group with large exponent}.

\begin{thm}[Schmid, Zhuang, {\cite[Theorem~1.2]{conjSrc}}]
\label{thm:conjbound}
    Let $p$ be an odd prime and let $G$ be a finite abelian $p$-group with $\D(G)\leq 2\exp(G)-1$. Then
    $$2\D(G)-1\leq \eta(G)+\exp(G)-1\leq \s(G)\leq \D(G)+2\exp(G)-2.$$
\end{thm}

Note that equality holds everywhere when $\D(G)=2\exp(G)-1$. Schmid and Zhuang conjectured that equality holds between the three leftmost expressions without this extra condition.

\begin{conj}[Schmid, Zhuang, {\cite[Conjecture~4.1]{conjSrc}}]
\label{conj:main}
    Let $G$ be a finite abelian $p$-group with $\D(G)\leq 2\exp(G)-1$. Then
    $$2\D(G)-1=\eta(G)+\exp(G)-1=\s(G).$$
\end{conj}

Gao, Han, and Zhang made progress on this conjecture in \cite{main}, proving the first equality under additional conditions on the rank and Davenport constant of $G$. Their result is stated here.

\begin{thm}[Gao, Han, Zhang, {\cite[Theorem~1.2]{main}}]
\label{thm:weakmain}
    Let $G$ be a finite abelian $p$-group with $\exp(G)=n$ and $\D(G)\leq 2n-1$. Write $G=C_{n}\oplus H$, and let $G'=C_{an}\oplus H$ for some $a\geq 1$. Then
    $$2\D(G')-1=\eta(G')+\exp(G')-1,$$
    provided that $p>2r(H)$ and $\left\lceil \frac{2\D(H)}{\exp(H)}\right\rceil$ is either even or at most $3$.
\end{thm}

Here taking $a=1$ yields the relevant progress towards Conjecture~\ref{conj:main}.

The second equality in the conjecture is a special case of a conjecture by Gao.

\begin{conj}[Gao, {\cite[Conjecture~2.3]{restrSize2}}]
\label{conj:etaEGZ}
    For any finite abelian group $G$, $$\eta(G)=\s(G)-\exp(G)+1.$$
\end{conj}

This conjecture holds for every $G$ for which $\eta(G)$ and $\s(G)$ have been determined, and is also known to hold in a few other special cases (see Theorem~6.6 in \cite{zssSurvey}).

In this paper, we prove the first equality in Conjecture~\ref{conj:main} in general.

\begin{thm}
\label{thm:main}
    Let $G$ be a finite abelian $p$-group with $\D(G)\leq 2\exp(G)-1$. Then
    $$2\D(G)-1=\eta(G)+\exp(G)-1.$$ 
\end{thm}

We also give an extension of Theorem~\ref{thm:main} to groups of the form $G'=C_a\oplus G$, where $G$ is a $p$-group with $\D(G)\leq 2\exp(G)-1$ and $a$ is not divisible by $p$. 

In the process of proving Theorem~\ref{thm:main}, we introduce some additional zero-sum invariants and techniques for using them, whose study may prove interesting in their own right or when applied to proving the second equality in Conjecture~\ref{conj:main}.

In the next section, we recall some useful previous results on zero-sum invariants in various settings. Afterwards, we give a simple proof of our main result and the extension. Finally, we extend the methods we use to make some progress on tackling the second half of Conjecture~\ref{conj:main} and discuss some ideas for further work.

\section{Background and Tools}

We start by giving some bounds on the invariants $\D(G)$, $\eta(G)$, and $\s(G)$ for general $G$.

\begin{thm}
For any finite abelian group $G$,

\[\D(G)\leq \eta(G)\leq \s(G)-\exp(G)+1\leq |G|.\]
\end{thm}

This result can be found in Lemma~5.7.2 and Theorem~5.7.4 of \cite{nonunique}. Conjecture~\ref{conj:etaEGZ} is just the claim that equality always holds in the second inequality.

Let $\D^*(G)=\sum_{i=1}^r (n_i-1) + 1$. A simple construction shows that $\D^*(G)\leq \D(G)$ for any $G$. We have already seen that equality holds for a group of rank at most two. Equality also holds for $p$-groups and in another case of interest.

\begin{thm}[Kruyswijk, Olson, {\cite{olson}}]
\label{thm:olson}
If $G$ is a $p$-group, then $\D(G)=\D^*(G)$.
\end{thm}

\begin{thm}[Geroldinger, {\cite[Corollary~4.2.13]{addgpNonunique}}]
\label{thm:nearD}
If $G'=C_a\oplus G$, where $G$ is a $p$-group with $\D(G)\leq 2\exp(G)-1$ and $a\in \NN^+$, then $\D(G')=\D^*(G')$.
\end{thm}

We will need the following result on short zero-sum subsequences of zero-sum sequences.

\begin{thm}[Fan, Gao, Wang, Zhong, Zhuang, {\cite[Theorem~2]{zssinzss}}]
\label{thm:prezssinzss}
    If $G=\C_{mp^n}\oplus H$ for some $m,n\geq 1$, where $H$ is a $p$-group with $\D(H)\leq p^n$, then any zero-sum sequence over $G$ of length at least $\D(G)+1$ contains a short zero-sum subsequence.
\end{thm}

We are particularly interested in the case $m=1$.

\begin{cor}
\label{thm:zssinzss}
    If $G$ is a $p$-group and $\D(G)\leq 2\exp(G)-1$, then any zero-sum sequence over $G$ of length at least $\D(G)+1$ contains a short zero-sum subsequence.
\end{cor}
\begin{proof}
A $p$-group $G$ with $\D(G)\leq 2\exp(G)-1$ can be written in the form $G=C_{p^n}\oplus H$, where $p^n=\exp(G)$ and, by Theorem~\ref{thm:olson}, $\D(H)=\D(G)-(p^n-1)\leq p^n$. So, the conclusion of Theorem~\ref{thm:prezssinzss} holds for such $p$-groups $G$.
\end{proof}

Finally, we will use the following tools to prove the extension of our main result.

\begin{lem}[Edel, Elsholtz, Geroldinger, Kubertin, Rackham, {\cite[Lemma~3.2]{zssBounds}}]
\label{lem:lowereta}
Let $H$ be a finite abelian group and $G'=C_{n}\oplus H$, where $\exp(H)\mid n$. Then
\[\eta(G')\geq 2(\D(H)-1)+n.\]
\end{lem}

\begin{prop}[Geroldinger, Halter-Koch, {\cite[Proposition~5.7.11]{nonunique}}]
\label{prop:inducteta}
If $H$ is a subgroup of $G$ such that $\exp(G)=\exp(H)\exp(G/H)$, then
\[\eta(G)\leq (\eta(H)-1)\exp(G/H)+\eta(G/H).\]
\end{prop}

Proposition~\ref{prop:inducteta} is an example of an application of the \emph{inductive method} \cite{nonunique}, an approach used to obtain bounds on the zero-sum invariants of a group $G$ in terms of the zero-sum invariants of its subgroups.

\section{Proof of the Main Result}
\subsection{The Invariants \texorpdfstring{$\zeta_i$}{z\_i}}
As the first step in the proof of Theorem~\ref{thm:main}, we define the invariants $\zeta_i(G)$ and compute them for a $p$-group $G$.
\begin{defn}
    Let $G$ be a finite abelian group and let $n=\exp(G)$. For $1\leq i\leq n$, let
    \[\zeta_i(G)=\s_{[i-n,0]+n\NN^+}(G).\]
\end{defn}

That is, $\zeta_i(G)$ is the shortest length $\ell$ for which any sequence over $G$ of length at least $\ell$ contains a nonempty zero-sum subsequence $T$ such that $T$ is congruent mod $\exp(G)$ to some residue between $i$ and $\exp(G)$, inclusive.

Note that $\zeta_1(G)=\D(G)$ by definition, and $\zeta_{\exp(G)}=\s_{\exp(G)\NN^+}(G)$. When $G$ is a $p$-group, the precise value of $\s_{\exp(G)\NN^+}(G)$ is known.

\begin{lem}[Gao, Geroldinger, {\cite[Theorem~6.7]{zssSurvey}}]
\label{lem:snN}
If $G$ is a $p$-group with $n=\exp(G)$, then
    \[\s_{n\NN^+}(G)=\D(G)+n-1.\]
\end{lem}

This allows us to compute $\zeta_i(G)$ for all $i\in[1,n]$ for a $p$-group $G$.

\begin{lem}
\label{lem:zetaVals}
Let $G$ be a $p$-group and let $n=\exp(G)$. For $1\leq i\leq n$, we have
\[\zeta_i(G)=\D(G)+i-1.\]
\end{lem}
\begin{proof}
Let $1\leq i\leq n-1$. For any $t>1$ and sequence $S$, if the sequence $0S$ obtained by prepending $0$ to $S$ contains a zero-sum subsequence of length $t$, then $S$ contains a zero-sum sequence of length $t-1$ or $t$. So $\zeta_i(G)\leq \zeta_{i+1}(G)-1$. Thus, by Lemma~\ref{lem:snN},
\[\D(G)=\zeta_1(G)<\zeta_2(G)<\cdots <\zeta_n(G)=\D(G)+n-1,\]
implying that $\zeta_i(G)=\D(G)+i-1$ for each $i$, as desired.
\end{proof}

Now we can proceed with the proof of Theorem~\ref{thm:main} itself.

\subsection{Main Proof}
\begin{proof}[Proof of Theorem~\ref{thm:main}]
Let $G$ be a $p$-group with $\D(G)\leq 2\exp(G)-1$. Let $n=\exp(G)$. By Theorem~\ref{thm:conjbound}, $\eta(G)\geq 2\D(G)-n$, so it suffices to show that $\eta(G)\leq 2\D(G)-n$.

Let $S$ be a sequence of length $2\D(G)-n$ over $G$. We wish to show that $S$ contains a short zero-sum subsequence. Note that $1\leq \D(G)-n+1\leq n$, so by Lemma~\ref{lem:zetaVals}, $|S|=\zeta_{\D(G)-n+1}(G)$. Thus, $S$ contains a nonempty zero-sum subsequence $T$ whose length $t$ is congruent mod $n$ to some residue between $\D(G)-n+1$ and $n$, inclusive. Then either $t\leq n$ or $t\geq \D(G)+1$.

If $t\leq n$, then $T\subseteq S$ is a short zero-sum subsequence of $S$, as desired.

Otherwise, $t\geq \D(G)+1$, so by Corollary~\ref{thm:zssinzss}, $T$ contains a short zero-sum subsequence, meaning $S$ contains a short zero-sum subsequence. So $S$ contains a short zero-sum subsequence in any case.

Thus, $\eta(G)= 2\D(G)-n$, which rearranges to the desired equality.
\end{proof}

\section{Extension to \texorpdfstring{$C_a\oplus G$}{C\_a+G}}

As with Theorem~\ref{thm:weakmain}, our main result can be extended to groups of the form $G'=C_a\oplus G$, where $G$ is a $p$-group with $\D(G)\leq 2\exp(G)-1$ and $a$ is not divisible by $p$. The proof is essentially the same as the derivation of Theorem~\ref{thm:weakmain} from the $a=1$ case, given in \cite{main}.

\begin{thm}
\label{thm:mainExt}
    Let $G$ be a finite abelian $p$-group with $\D(G)\leq 2\exp(G)-1$. Then for any positive integer $a$ not divisible by $p$, $G'=C_a\oplus G$ satisfies
    $$2\D(G')-1=\eta(G')+\exp(G')-1.$$ 
\end{thm}
\begin{proof}
Let $n=\exp(G)$. By Theorem~\ref{thm:nearD} we have $\D(G')=\D^*(G')=\D(G)+(a-1)n$, and by Lemma~\ref{lem:lowereta} we have $\eta(G')\geq 2(\D(G)-n)+an=2\D(G')-an$.

By Proposition~\ref{prop:inducteta}, we have
\[\eta(G')\leq (\eta(C_a)-1)\exp(G)+\eta(G)=(a-1)n+\eta(G).\]

Finally, by Theorem~\ref{thm:main}, we have $\eta(G)=2\D(G)-n$, so
\[\eta(G')\leq 2\D(G)+(a-2)n=2\D(G')-an.\]
Thus $\eta(G')=2\D(G')-\exp(G')$, which rearranges to the desired equality.
\end{proof}

Note that Theorem~\ref{thm:mainExt} is an extension of Theorem~4.2.1 in \cite{congCond}. Remark~4.3 in \cite{congCond} pointed out that a result like this would follow given Theorem~\ref{thm:main}.

\section{Towards Evaluating \texorpdfstring{$\s(G)$}{s(G)}}

It would be interesting to see if a similar method suffices for verifying the second half of Conjecture~\ref{conj:main}, showing that $\s(G)=2\D(G)-1$ for $G$ a $p$-group with $\D(G)\leq 2\exp(G)-1$. Here we describe an approach that seems promising and use it to make some progress.

Let $n=\exp(G)$. For $1\leq i\leq n$, let $\eta_i(G)=\s_{[i,n]}(G)$. Since a sequence $S$ contains a zero-sum subsequence of length $i$ or $i+1$ if $0S$ contains a zero-sum subsequence of length $i+1$, it is easy to see that
\[\eta(G)=\eta_1(G)<\cdots<\eta_n(G)=\s(G).\]

So, if Conjecture~\ref{conj:etaEGZ} holds in general, $\eta_i(G)=\eta_{i-1}(G)+1=\eta(G)+i-1$ should always hold for $2\leq i\leq n$.

The following partial result is known.

\begin{lem}[Gao, {\cite[Lemma~2.6]{restrSize2}}]
\label{lem:etaihalf}
For any finite abelian group $G$, $\eta_i(G)=\eta(G)+i-1$ for $2\leq i\leq \lfloor \frac{n}{2}\rfloor + 1$.
\end{lem}

Unfortunately, the approach used to prove this lemma does not directly generalize to the case where $i>\lfloor \frac{n}{2}\rfloor + 1$. However, we can apply the same idea to get an improved result when specializing to a $p$-group of large exponent, via a result generalizing Corollary~\ref{thm:zssinzss} on short zero-sum subsequences of zero-sum sequences.

\begin{thm}
\label{thm:zssin2n}
Let $G$ be a $p$-group with $\D(G)\leq 2\exp(G)-1$. For $1\leq i \leq 2\exp(G)-\D(G)$, any zero-sum sequence $S$ of length exactly $\D(G)+i$ over $G$ contains a short zero-sum subsequence of length at least $i$. 
\end{thm}
\begin{proof}
We use strong induction on $i$. The base case $i=1$ is Corollary~\ref{thm:zssinzss}. Assume the result is true for $1\leq i\leq j-1$, and take $i=j\leq 2\exp(G)-\D(G)$.

Let $n=\exp(G)$ and consider a zero-sum sequence $S$ of length $\D(G)+j$ over $G$. Let $S'$ be the subsequence of $S$ consisting of all but the last element, so $S'$ has length $|S'|=\D(G)+j-1$. By Lemma~\ref{lem:zetaVals}, $S'$ contains a zero-sum sequence $T$ whose length $|T|$ is congruent $\mod n$ to some residue between $j$ and $n$, inclusive. Since $\D(G)+j\leq 2n$, either $j\leq |T|\leq n$ or $n+j\leq |T|\leq 2n-1$. In the former case, $T$ is a short zero-sum subsequence of $S$ of length at least $j$, as desired. In the latter case, the complement $T'$ of $T$ in $S$ is a nonempty zero-sum subsequence of $S$ of length $|T'|\leq (\D(G)+j) - (n+j) = \D(G)-n\leq n-j$.

If $|T'|\geq j$, we again have a short zero-sum subsequence of $S$ of length at least $j$, as desired. Otherwise, $|T|=\D(G)+j-|T'|>\D(G)$. By the inductive hypothesis, $T$ contains a short zero-sum subsequence $U$ of length at least $j-|T'|$. If $|U|\geq j$, we are again done, so assume $j-|T'|\leq |U|\leq j-1$. Then
\[j\leq |T'|+|U|\leq n-1,\]
so $T'U$ is a short zero-sum subsequence of $S$ of length at least $j$, as desired. So the result holds for $i=j$ as well, and thus by induction holds for all $i\leq 2\exp(G)-\D(G)$.
\end{proof}

This gives the following extension of Theorem~\ref{thm:main} as an immediate corollary.

\begin{cor}
\label{cor:etaimain}
    Let $G$ be a finite abelian $p$-group with $\D(G)\leq 2\exp(G)-1$. Then for $1\leq i\leq \max\left(2\exp(G)-\D(G), \left\lfloor\frac{\exp(G)}{2}\right\rfloor + 1\right),$ we have
    $$\eta_i(G)=2\D(G)-\exp(G)+(i-1).$$ 
\end{cor}
\begin{proof}
Let $n=\exp(G)$, and let $R$ be a sequence of length $\D(G)+n-1$ over $G$. Since $\D(G)+n-1<3n$, by Lemma~\ref{lem:snN}, $R$ contains a zero-sum subsequence $S$ of length $n$ or $2n$. If $|S|=2n$, then Theorem~\ref{thm:zssin2n} gives a short zero-sum subsequence $T$ of length at least $2n-\D(G)$. So, in either case, $R$ contains a short zero-sum subsequence of length at least $2n-\D(G)$, meaning $\eta_{2n-\D(G)}(G)\leq \D(G)+n-1$. Since $\eta(G)=2\D(G)-n$, the bounds $\eta_{i-1}(G)<\eta_i(G)$ yield the desired result for all $i\leq 2\exp(G)-\D(G)$.

The result for $i\leq \left\lfloor\frac{n}{2}\right\rfloor + 1$ is a direct consequence of Theorem~\ref{thm:main} and Lemma~\ref{lem:etaihalf}.
\end{proof}

The argument used in Theorem~3.1 of \cite{conjSrc} to prove that $\s(G)\leq \D(G)+2n-2$ for a $p$-group $G$ of large exponent relies on the fact that a sequence $S$ over $G$ of length at least $\D(G)+(n-2)+\ell$ contains either a zero-sum subsequence of length $n$, or a zero-sum subsequence of length $2n$ that itself contains a short zero-sum subsequence of length at least $\ell$. Our Theorem~\ref{thm:zssin2n} provides a strengthening of this result for $\ell=1$, giving a short zero-sum subsequence of length at least $(2n-1)-\D(G)+\ell=2n-\D(G)$ in the zero-sum subsequence of length $2n$. It is conceivable that this strengthening could generalize to larger values of $\ell$, as the following conjecture voices.

\begin{conj}
Let $G$ be a finite abelian $p$-group with $\exp(G)=n$ and $\D(G)\leq 2n-1$. Let $\ell\in [1,\D(G)+1-n]$, and let $S$ be a sequence over $G$ of length at least $\D(G)+(n-2)+\ell$. Then one of the following statements holds:
\begin{enumerate}
    \item $S$ has a zero-sum subsequence $B$ with $|B|=n$.
    \item $S$ has a zero-sum subsequence $B$ with $|B|=2n$ such that $B$ has a zero-sum subsequence $B'$ with $|B'|\in [(2n-1)-\D(G)+\ell,n-1]$.
\end{enumerate}
\end{conj}

This conjecture, if true, would yield the desired bound of $\s(G)\leq \D(G)+(n-2)+(\D(G)-n+1)=2\D(G)-1$. Perhaps in this or some other way, the argument used in \cite{conjSrc}, extended by the arguments we use, could lead to the full resolution of Conjecture~\ref{conj:main}.

\section*{Acknowledgements}
This research was conducted at the University of Minnesota Duluth REU and was supported by NSF grant 1358659 and NSA grant H98230-16-1-0026. The author thanks Joe Gallian for suggesting the problem and Levent Alpoge, Alfred Geroldinger, Wolfgang Schmid, and Qinghai Zhong for helpful comments on the manuscript.

\bibliographystyle{acm}

\bibliography{main}

\begin{thebibliography}{10}

\bibitem{zssBounds}
{\sc Edel, Y., Elsholtz, C., Geroldinger, A., Kubertin, S., and Rackham, L.}
\newblock Zero-sum problems in finite abelian groups and affine caps.
\newblock {\em Q. J. Math. 58}, 2 (2007), 159--186.

\bibitem{zssinzss}
{\sc Fan, Y., Gao, W., Wang, G., Zhong, Q., and Zhuang, J.}
\newblock On short zero-sum subsequences of zero-sum sequences.
\newblock {\em Electron. J. Combin. 19}, 3 (2012), P31.

\bibitem{restrSize2}
{\sc Gao, W.}
\newblock On zero-sum subsequences of restricted size {II}.
\newblock {\em Discrete Mathematics 271}, 1 (2003), 51--59.

\bibitem{zssSurvey}
{\sc Gao, W., and Geroldinger, A.}
\newblock Zero-sum problems in finite abelian groups: a survey.
\newblock {\em Expo. Math. 24}, 4 (2006), 337--369.

\bibitem{main}
{\sc Gao, W., Han, D., and Zhang, H.}
\newblock The {EGZ}-constant and short zero-sum sequences over finite abelian
  groups.
\newblock {\em J. Number Theory 162\/} (2016), 601--613.

\bibitem{congCond}
{\sc Geroldinger, A., Grynkiewicz, D.~J., and Schmid, W.~A.}
\newblock Zero-sum problems with congruence conditions.
\newblock {\em Acta Math. Hungar. 131}, 4 (2011), 323--345.

\bibitem{nonunique}
{\sc Geroldinger, A., and Halter-Koch, F.}
\newblock {\em {N}on-unique {F}actorizations: {A}lgebraic, {C}ombinatorial and
  {A}nalytic {T}heory}.
\newblock Chapman \& Hall/CRC Pure and Applied Mathematics. CRC Press, 2006.

\bibitem{addgpNonunique}
{\sc Geroldinger, A., and Ruzsa, I.}
\newblock {\em {C}ombinatorial {N}umber {T}heory and {A}dditive {G}roup
  {T}heory}.
\newblock Advanced Courses in Mathematics - CRM Barcelona. Birkh{\"a}user
  Basel, 2009.

\bibitem{olson}
{\sc Olson, J.~E.}
\newblock A combinatorial problem on finite abelian groups, {II}.
\newblock {\em J. Number Theory 1}, 2 (1969), 195--199.

\bibitem{conjSrc}
{\sc Schmid, W.~A., and Zhuang, J.}
\newblock On short zero-sum subsequences over p-groups.
\newblock {\em Ars Combin. 95\/} (2010), 343--352.

\bibitem{zsramsey}
{\sc Zhang, H., and Wang, G.}
\newblock On the {E}rd{\H{o}}s--{G}inzburg--{Z}iv invariant and zero-sum
  {R}amsey number for intersecting families.
\newblock {\em Int. J. Number Theory 10}, 07 (2014), 1637--1647.

\end{thebibliography}

\end{document}